\documentclass{amsart}

\usepackage{latexsym,amssymb,amsthm,amsmath,amscd,graphicx}

\theoremstyle{plain}
\newtheorem{theorem}{Theorem}[section]

\newtheorem{lemma}{Lemma}[section]
\newtheorem{proposition}{Proposition}[section]

\newtheorem{corollary}{Corollary}[section]

\newtheorem{example}{Example}[section]
\numberwithin{equation}{section}

\theoremstyle{remark}
\newtheorem{remark}{Remark}[section]

 \numberwithin{equation}{section}

\newtheorem*{Theorem A}{{\bf Theorem A}}
\newtheorem*{Theorem B}{{\bf Theorem B}}
\newtheorem*{Theorem C}{Theorem C}

 \numberwithin{equation}{section}
\def\<{\left < }
\def\>{\right >}
\def\({\left ( }
\def\){\right )}

\def\r{\eqref }

\begin{document}

\title[Generalized $1$-harmonic Equation and The Inverse Mean Curvature Flow] {On A Generalized $1$-harmonic Equation and The Inverse Mean Curvature Flow}

\author[Y.I. Lee,  A.N. Wang, and S.W. Wei]{Yng-Ing Lee, Ai-Nung Wang, and Shihshu Walter Wei$^*$}
\address{Department of Mathematics and Taida Institute for Mathematical Sciences\\
National Taiwan University, Taipei, Taiwan \\
National Center for Theoretical Sciences, Taipei Office}
\email{yilee@math.ntu.edu.tw}
\address{Department of Mathematics\\
National Taiwan University\\ Taiwan} \email{wang@math.ntu.edu.tw}
\address{Department of Mathematics\\
University of Oklahoma\\ Norman, Oklahoma 73019-0315\\
U.S.A.}
\email{wwei@ou.edu}
\thanks{$^*$ Wishes to thank Taida Institute for Mathematical Sciences
and the Department of Mathematics at National Taiwan University for their kind invitation, support and hospitality. This research was partially supported by NSF Award No DMS-0508661, the OU Presidential International Travel Fellowship, and the OU
Faculty Enrichment Grant.\\ }

\begin{abstract} We {\it introduce} and study generalized $1$-harmonic equations \r{1.1}. Using some ideas and techniques in studying $1$-harmonic functions from \cite
{W1} (2007), and in studying nonhomogeneous $1$-harmonic functions on a cocompact set from 
[W2, (9.1)] (2008), we find an {\it analytic} quantity $w$ in the generalized
$1$-harmonic equations \r{1.1} on a domain in a Riemannian
$n$-manifold that affects the behavior of weak solutions of
\r{1.1},  and establish its link with the {\it geometry} of the
domain. We obtain, as applications, some gradient bounds and
nonexistence results for the inverse mean curvature flow,
Liouville theorems for $p$-subharmonic functions of constant
$p$-tension field, $p \ge n\, ,$ and nonexistence results for solutions of the initial value problem of inverse mean curvature flow.
\end{abstract}


\date{}
\maketitle
\section{Introduction}
Some ideas and techniques in studying $1$-harmonic functions from
\cite {W1}, and in studying nonhomogeneous $1$-harmonic functions on a cocompact set from
[W2, (9.1)], can be carried over to a more general setting, e.g. to a
large class of the following partial differential equations:
\begin{equation}\label{1.1} \operatorname{div} \bigg(\frac {\nabla f} { |\nabla f|}\bigg)= A(x,
f, \nabla f)\, \ \ on \ \ \Omega,\end{equation} where $A(x,
f(x), \nabla f(x))$ is a
continuous real-valued function of $x \in \Omega\, ,$ or more generally $A(x,
f(x), \nabla f(x))$ is a nonnegative-valued or nonpositive-valued function,
and $\Omega$ is a domain in a complete Riemannian $n$-manifold
$M$. These equations, in particular, include the ones for
{\it $1$-harmonic functions} when $A(x, f, \nabla f)\equiv 0\, ,$
{\it functions of constant $1$-tension field} when $A(x, f, \nabla
f)\equiv \operatorname{const}\, ,$ {\it nonhomogeneous $1$-harmonic functions} when $A(x, f, \nabla
f) = A(x)\, ,$ {\it almost $1$-harmonic functions} when $A(x, f, \nabla
f) = A(x, f)\, ,$ {\it the mean curvature flow} when $A(x,
f, \nabla f) = -\frac 1 {|\nabla f|}\, ,$ and {\it the inverse mean curvature flow} when $A(x,
f, \nabla f) = |\nabla f|\, $ (in the level set formulation, where the evolving hypersurfaces $X_t$ are given as level sets of $f\, ,$ via $X_t = \partial \{x: f(x) < t\}$).
Whereas $1$-harmonic functions have been applied to solve
Plateau's problem in Euclidean space (\cite{BDG}) and the Bernstein conjecture in hyperbolic geometry  (\cite{WW}), the
inverse mean curvature flow (i.e., a solution of the parabolic evolution equation $\frac {\partial
X_t}{\partial t} = \frac {\nu}{H}$, where $H\, ,$ assumed to be
positive, is the mean curvature of $X_t$, $\nu$
is the outward unit normal, and $\frac {\partial X_t}{\partial t}$
denotes the normal velocity field along $X_t$) has been applied to
solve fundamental problems in general relativity, such as proving the
Penrose inequality (\cite {HI}), and to compute the Yamabe
invariant of three-dimensional real projective space (\cite {BN}).
Thus, in view of numerous relations among $p$-harmonic maps,
geometric flows, and other areas of mathematics and science (e.g. cf. \cite {HI,M,KN,BN,BDG,WW,W1,W2,D,KS,ES,CGG}), we
would like to call equations of the form \r{1.1} {\it generalized
$1$-harmonic equations}, and initiate their study in this paper. \smallskip

A $W^{1,1}_{loc}(\Omega)$ function $f: \Omega \to \mathbb{R}$ is
said to be {\it a weak solution of $\r{1.1}$} in the distribution
sense, or {\it a generalized $1$-harmonic function}, if for every
$\psi(x) \in C_0^{\infty}(\Omega)\, ,$ we have
\begin{equation}\begin{aligned}\label{1.2} & \int _{\Omega} \frac {\nabla f} { |\nabla
f|}\cdot \nabla \psi dv   = -  \int _{\Omega} A(x, f(x), \nabla
f(x))\psi(x) dv\, .\end{aligned}\end{equation} We note that in the
case of the inverse mean curvature flow, our definition of a weak
solution $f\in W^{1,1}_{loc}(\Omega)$ is a {\it critical point} of the functional $J^K_f(g) = \int _K
|\nabla g| + g |\nabla f| dv\,$ with fixed $|\nabla f|$ (where $g \in
W^{1,1}_{loc}(\Omega),$ and $K \subset \subset \Omega)\, ,$ in the distribution sense. This is in contrast to that used
in G. Huisken and T. Ilmanen \cite {HI}. According to the
definition used in \cite {HI}, $f\in C^{0,1}_{loc}(\Omega)(=
W^{1,\infty}_{loc}(\Omega))$ is a weak solution of \r{1.8} if $f$
is an {\it absolute minimum} of the functional $J^K_f(g) = \int _K
|\nabla g| + g |\nabla f| dv\,$ (where $g \in
C^{0,1}_{loc}(\Omega),$ and $K \subset \subset \Omega)\, .$ By
this definition, a constant function $f$ would be a weak solution
of \r{1.8}, whereas constant functions are not admissible as weak
solutions of \r{1.1} by our definition. However, in both definitions, weak solutions are not in the  space of functions of
bounded variation for
problems of this type, and in general they are closely related by the inclusions $C^{0,1}_{loc}(\Omega)\subset
W^{1,1}_{loc}(\Omega)\, .$   \smallskip

  The purpose of this paper is to find an {\it analytic} quantity $w$ (cf.\r{1.3}) in a generalized
$1$-harmonic equation \r{1.1} on a domain $\Omega$ in a complete
Riemannian manifold $M$ that affects the behavior of weak
solutions of \r{1.1}, and establish its link with the {\it
geometry} of the domain $\Omega$, in terms of its doubling
constant $D(\Omega)$, Cheeger constant  $I_{\infty}(\Omega)$, Sobolev constant $S_{\infty}(\Omega)$ and the first
eigenvalue $\lambda _1(\Omega)$ of the Laplacian. In particular, we have: \smallskip

\noindent {\bf Theorem 2.1}$\quad$
{\it Let $\Omega \subset M$ have the doubling property \r{2.1}.
Assume $f\in W^{1,1}_{loc}(\Omega)$ is a
 weak solution of \r{1.1}. Let
\begin{equation}\label{1.3}w =
ess \inf_{x\in \Omega} |A(x, f(x), \nabla f(x))|\, .\end{equation}
  Then for every $B(x_0,r)\subset
\Omega$, the essential infimum $w$ of $|A|$ over $\Omega$ satisfies
\begin{equation}\label{1.4}0 \le w \le \frac {C_1 D(\Omega)}r\,
,\end{equation}  where constant $C_1\, $ is as in \cite {W1} Lemma 1, and
$D(\Omega)$ is a doubling constant of $\Omega$ as in \r{2.1}.}
\smallskip

\noindent {\bf Theorem 2.2}$\quad$
{\it Let  $\Omega \subset M$, the Cheeger constant $I_{\infty}(\Omega)$
 be as defined in $\r{2.3}$, and $w$ be as in \r{1.3}. Then
\begin{equation}\label{1.5}0 \le w \le I_{\infty}(\Omega),\end{equation}
and for all $x_0 \in \Omega$ and $B(x_0,r)\subset \Omega$, $w$
satisfies \begin{equation}\label{1.6} w {\rm Vol} (B(x_0,r)) \le \frac
{d}{dr} {\rm Vol} (B(x_0,r))\, \end{equation} for almost all $r > 0$. In
particular, if $I_{\infty}(\Omega) = 0$, then $w=0\, .$}\smallskip

Examples include that on a complete manifold with nonnegative
Ricci curvature, the essential infimum $w$ of $A$ over $\Omega$
satisfies \r{1.4}, on a complete manifold with negative sectional
curvature $Sec^M \le - a^2\ < 0$, $w$ satisfies \r{1.5} and
\r{1.6}, and in Euclidean space $\mathbb{R}^n, w=0$.\smallskip

As applications of our methods, we obtain some gradient bounds and
nonexistence results for \emph {the inverse mean curvature flow}:

\noindent {\bf Corollary 3.1}$\quad$
{\it Let $\Omega$ have a doubling constant and $f: \Omega \to \mathbb{R}$
be a $W^{1,1}_{loc}$ weak solution of the level set formulation of
the inverse mean curvature flow
\begin{equation}\label{1.8} \operatorname{div} \bigg(\frac {\nabla f}
{ |\nabla f|}\bigg)= | \nabla f |\, \ \ on \ \ \Omega,\end{equation}
Then \begin{equation}\label{1.9} ess \, \inf_{x \in B(x_0,r)} |\nabla f|(x) \le
\frac{C_1D(\Omega)}r\, ,\end{equation} for any $x_0 \in \Omega,$ and
$B(x_0,r)\subset \Omega$.}

\noindent {\bf Corollary 3.2}$\quad$
{\it Let  $\Omega \subset M$, the
Sobolev constant $S_{\infty}(\Omega)$ be
as defined in $\r{2.4}$, and $f\in W^{1,1}_{loc}(\Omega)$ be a
weak solution of \r{1.8}; then
\begin{equation}\label{1.10}ess \, \inf_{x \in \Omega} |\nabla f|(x)  \le
I_{\infty}(\Omega)\quad \text{and} \quad ess \, \inf_{x \in
\Omega} |\nabla f|(x)  \le S_{\infty}(\Omega).\end{equation}
and
\begin{equation}\label{1.12}ess
\, \inf_{x \in \Omega} |\nabla f|(x) \le 2 \sqrt {\lambda
_1(\Omega)}.\end{equation}}

\noindent {\bf Proposition 3.2}$\quad$
{\it Let $M$ be a complete noncompact manifold with the global doubling
property.  Then there \emph {does not exist} a $W^{1,1}_{loc}$ weak solution of the
level set formulation of the inverse mean curvature flow \r{1.8}
on $M$ that satisfies \begin{equation}\label{1.11}\lim_{r \to \infty} ess
\, \inf_{x \in B(x_0,r)} |\nabla f|(x) > 0\, ,\end{equation}
for some $x_0 \in M$.
}

By the
same method, we also obtain Liouville theorems for $p$-subharmonic
functions of constant $p$-tension field:

\noindent {\bf Theorem 4.1}$\quad$
{\it Let $n \le p\, ,$ and $f\in W^{1,p}_{loc} (\mathbb{R}^n) \cap
C(\mathbb{R}^n)$ be bounded below. If $f$ is a weak $\emph
{subsolution}$ of the $p$-Laplace equation with constant $p$-tension
field, i.e., $div(|\nabla f|^{p-2}\nabla f) = c, \,$ and with
bounded weak derivative (i.e., $|\nabla f| \le C_2$ for some
constant $C_2
> 0$), then $f$ is constant.}

One can drop the assumption on the bounded gradient if the
function is in $W^{1,p} (\mathbb{R}^n)\, :$

\noindent {\bf Theorem 4.2}$\quad$
{\it Let $n \le p\, .$ If $f\in W^{1,p} (\mathbb{R}^n) \cap
C(\mathbb{R}^n)$ is bounded below, and is a weak $\emph
{subsolution}$ of the $p$-Laplace equation with constant $p$-tension
field, then $f$ is constant.}

These augment Theorem 1.1 in \cite {BD}(cf. Theorem \ref{T:4.1})
in which a weak $\emph {supersolution}$ of the $p$-Laplace equation
was considered. It should  also be mentioned that a positive
$p$-subharmonic (resp. $p$-superharmonic) function $f$ on a
complete noncompact Riemannian manifold with one of the following:
\emph{$p$-finite}, \emph{$p$-mild}, \emph{$p$-obtuse},
\emph{$p$-moderate}, and \emph{$p$-small} growth, in which $1 < p <
\infty$, and the growth exponent $q
> p-1 \ \ (resp. \ \ q < p-1)\, ,$ is constant (\cite {WLW}).
As a further application, we obtain  in particular, for $p=1$,
that if a complete noncompact Riemannian manifold $M$ has the
global doubling property, and $f \in W^{1,1}_{loc}(M)$ is a weak
subsolution of the $1$-harmonic equation on $M$ of a constant
$1$-tension field, then $f$ is constant (cf. Proposition
\ref{P:4.1}). This generalizes the case $M = \mathbb{R}^n$ in
\cite {W1}.

Being motivated by the work in \cite {W2}(cf. Theorem \ref{T:5.0}), we study generalized $1$-harmonic functions on cocompact domains in Section 5, and obtain the nonexistence of solutions of the initial value problem for inverse mean curvature flow
in Section 6 (cf. Theorems \ref{T:5.1} and \ref{T:6.2}).

Our methods developed in \cite {W1,W2} and in this paper can be employed and carried over to other settings, such as {\it generalized constant mean curvature type equations for differential forms} in Euclidean space and on manifolds \cite{DW}, and nonhomogeneous $A$-harmonic equations \cite {L}.

We wish to thank the editor and the referee for their comments and suggestions, and Professor Junfang Li for his remarks that have made the present form of the paper possible.

\section{Generalized $1$-harmonic equations}

We recall that a domain $\Omega \subset M$ is said to have the {\it
doubling property} if $\exists D(\Omega)
> 0\, $ s.t. $\forall r
> 0\, ,$ $\forall x \in M\, ,$ with $B(x,2r) \subset \Omega \,
, $ the volumes of the geodesic balls centered at $x$ of radii
$2r$ and $r\, $ satisfy
\begin{equation}\label{2.1}{\rm Vol}(B(x,2r)) \le D(\Omega) \ \ {\rm Vol}(B(x,r))\, .\end{equation}
$M$ is said to have the {\it global doubling property} if
$\exists D(M) > 0\, $ s.t. $\forall r
>0\, ,$ $\forall x \in M\, ,$
\begin{equation}\label{2.2}{\rm Vol}(B(x,2r)) \le D(M) \ \ {\rm Vol}(B(x,r))\, .\end{equation}
For our convenience, we call $D(\Omega)$ (resp. $D(M)$) a doubling
constant of $\Omega$ (resp. $M$).\bigskip

\begin{theorem}\label{T:1.1}
Let $\Omega \subset M$ have the doubling property \r{2.1}.
Assume that $f\in W^{1,1}_{loc}(\Omega)$ is a
 weak solution of \r{1.1}. Let the essential infimum $w$ of $|A|$ over $\Omega$ be defined as in \r{1.3}.
  Then for every $B(x_0,r)\subset
\Omega$, $w$ satisfies \r{1.4}.
\end{theorem}

\begin{proof} \quad   We consider two cases:

Case 1. $A(x, f(x), \nabla
f(x))$ is continuous on $\Omega$ and assumes both positive and negative values: By the
intermediate value theorem, $A(x, f(x), \nabla
f(x))$ assumes value $0$ at some point,
and thus $w=\inf_{x\in \Omega}|A(x, f(x), \nabla
f(x))|=0$.

Case 2. $A(x, f(x), \nabla
f(x))$ is nonpositive-valued or nonnegative-valued: Let $\psi \ge 0$ be
as in \cite {W1} Lemma 1, in which $t=r, s= \frac r2$. Substituting
$\psi$ into \r{1.2}, applying \r{1.3} and the Cauchy-Schwarz
inequality, we have

$$
\aligned 0 & \le w {\rm Vol}(B(x_0,\frac r2)) \\
& = \int _{B(x_0,\frac r2)}
w \psi(x) dx  \\
& \le \bigg|\int _{B(x_0,\frac r2)} A(x, f(x), \nabla
f(x))\psi(x) dv \bigg|\\
& \le \bigg|\int _{B(x_0,r)} A(x, f(x), \nabla
f(x))\psi(x) dv\bigg|  \\& =  \bigg|\int _{B(x_0,r)} \frac {\nabla
f} { |\nabla f|}\cdot \nabla \psi dv\bigg| \\ & \le  \int _{B(x_0,r)} |\nabla
\psi| dv \\
& \le \frac{C_1}{r}{\rm Vol}
(B(x_0,r))\\
& \le \frac{C_1 D(\Omega)}{r}{\rm Vol}
(B(x_0,\frac r2))\endaligned
$$
\end{proof}

By letting $r \to \infty$ in the above expression, we obtain immediately the following

\begin{corollary}\label{C:2.1}
Let $M$ be a complete noncompact manifold with a doubling
constant. Then there does not exist a $W^{1,1}_{loc}$ weak
subsolution $f: M \to \mathbb{R}$ of equation \r{1.1} on $M\, $
such that for some $x_0 \in M\, ,$ $$\lim_{r\to \infty} ess \,
\inf_{x \in B(x_0,r)} |A(x, f(x), \nabla f(x))| > 0\, .$$
\end{corollary}

Define the {\it Cheeger constant} $I_{\infty}(\Omega)  $ by
\begin{equation}\label{2.3} I_{\infty}(\Omega) = \sup \{ C: {\rm Area} (\partial \Omega^{\prime}) \ge C \,
{\rm Vol}(\Omega^{\prime}) \ \ for \ \ any \ \ domain \ \
\Omega^{\prime} \subset \subset \Omega\}\end{equation}

\begin{theorem}\label{T:1.2}
Let  $\Omega \subset M$, the Cheeger constant $I_{\infty}(\Omega)$
 be as defined in $\r{2.3}$, and $w$ be as in \r{1.3}. Then $w$ satisfies \r{1.5}, and for all $x_0 \in \Omega$ and $B(x_0,r)\subset \Omega$, $w$
satisfies \r{1.6} for almost all $r > 0$. In
particular, if $I_{\infty}(\Omega) = 0$, then $w=0\, .$
\end{theorem}

\begin{proof} \quad By \r{1.1}, \r{1.3}, Stokes' theorem, and the
Cauchy-Schwarz inequality, we have, for any domain $
\Omega^{\prime} \subset \subset \Omega$  \[
\begin{array}{rll}  w {\rm Vol}(\Omega^{\prime})  & \le  & \bigg| \int _{\Omega^{\prime}}
A(x, f(x), \nabla f(x)) dx \bigg|= \bigg|\int _{\partial \Omega^{\prime}}
\frac {\nabla f} { |\nabla f|}\cdot \nu dS\bigg| \\ & \le & \int
_{\partial \Omega^{\prime}} 1 dS
=  {\rm Vol} (\partial \Omega^{\prime}) \, ,
\end{array}
\]
This yields \r{1.5} immediately, and \r{1.6} by the coarea formula and
letting $\Omega^{\prime} = B(x_0,r)$.
\end{proof}


Define a {\it Sobolev constant} $S_{\infty}(\Omega)$ by

\begin{equation}\label{2.4}S_{\infty}(\Omega) = \sup \{ C: C \int _{\Omega} |f| dv \le \int _{\Omega} |\nabla f| dv \,
 \ \ for \ \ any \ \ f \in C_0^{\infty}(\Omega)\}\end{equation}
 Then one always has:

\begin{corollary}\label{C:2.2} $$w \le S_{\infty}(\Omega)\, .$$
In particular, if $S_{\infty}(\Omega)=0\, ,$ then $w =0\, .$

\end{corollary}
\begin{proof} This follows from a theorem of Federer and Fleming
(\cite {FF}) that $S_{\infty}(\Omega)  = I_{\infty}(\Omega) \, ,$
and Theorem \ref{T:1.2}.
\end{proof}

\begin{corollary}\label{C:2.3}
Under the assumption of Theorem \ref{T:1.2}, or Corollary
\ref{C:2.2}, we have
$$w \le 2 \sqrt {\lambda _1(\Omega)}$$ where $\lambda _1(\Omega)$
is the first eigenvalue of the Laplacian on $\Omega$.
\end{corollary}

\begin{proof} This follows from the inequality
$\lambda _1(\Omega) \ge \frac 14 I_{\infty}(\Omega)^2\, $(cf. \cite {C}),  and Theorem
\ref{T:1.2}.
\end{proof}

\begin{proposition}\label{P:2.1}
 Let $M$ be a complete noncompact Riemannian manifold, and $w$ be as in \r{1.3} in which $\Omega = M$.
 Then, for every $x_0 \in M\, ,$ there exists a constant $K$ such that  $w$ satisfies
 \begin{equation}\label{2.5} \lim \inf _{r\to \infty} e^{-wr} {\rm Vol} (B(x_0,r)) \ge K > 0\, ,\end{equation}
 In particular, if $M$ has $p$th-power volume growth, for some $p \ge 0$  $($i.e., ${\rm Vol} (B(x_0,r)) = O (r^p)$ as $r \to \infty$ for some $x_0 \in M$$)$, then $w =0\, .$
\end{proposition}
\begin{proof} Let $r_0$ and $r_1$ be any two numbers with $r_1 > r_0 > 0\, .$ Integrating \r{1.6} with respect to $r$ from $r_0$ to $r_1\, ,$ we have
\begin{equation}\begin{aligned}\notag & wr_1 - wr_0  = \int _{r_0}^{r_1} w\, dr \le \int _{r_0}^{r_1} \frac {\frac {d}{dr}{\rm Vol} (B(x_0,r))}{{\rm Vol} (B(x_0,r))} dr   \\
&\hskip.3in =  \log {\rm Vol} (B(x_0,r_1)) - \log {\rm Vol}
(B(x_0,r_0))\, ,\end{aligned}\end{equation} for any $r_1
> r_0\, .$ This
implies, by the monotonicity of the exp function, ${e^{wr_1}}\cdot
{e^{- wr_0}}\le {\rm Vol} (B(x_0,r_1))\cdot ({\rm Vol}
(B(x_0,r_0)))^{- 1}\, ,$ i.e., $e^{- wr_1}{\rm Vol} (B(x_0,r_1))
\ge K > 0\, ,$ for any $r_1
> r_0\, ,$ where $K = e^{- cr_0}{\rm Vol} (B(x_0,r_0))\, .$ Now taking
$r_1$ to $\infty$ gives the desired \r{2.5}. Suppose
the contrary for the last assertion; then \r{2.5} would lead to $0 =
\lim \inf _{r\to \infty} e^{-wr} r^p >0\, ,$ a contradiction.
\end{proof}

\section{The inverse mean curvature flow}

For completeness and the readers' reference,  we collect the
consequences of our observation on the inverse mean curvature flow
in this section.
\begin{corollary}\label{C:1.1}
Let $\Omega$ have a doubling constant, and $f: \Omega \to \mathbb{R}$
be a $W^{1,1}_{loc}$ weak solution of the level set formulation of
the inverse mean curvature flow \r{1.8}.
Then $ess \, \inf_{x \in B(x_0,r)} |\nabla f|(x) \, $ satisfies \r{1.9}, for any $x_0 \in \Omega,$ and
$B(x_0,r)\subset \Omega$.
\end{corollary}

\begin{proof} \quad This follows
at once from Theorem \ref{T:1.1} in which $A(x, f(x), \nabla f(x))
= |\nabla f(x)|.$\end{proof}

\begin{proposition}\label{P:3.1}
Let $M$ be a complete noncompact manifold with a doubling
constant. Let $f: M \to \mathbb{R}$ be a $W^{1,1}_{loc}$ weak
subsolution of the level set formulation of the inverse mean
curvature flow \r{1.8} on $M$. Then $$\lim_{r \to \infty} ess \,
\inf_{x \in B(x_0,r)} |\nabla f|(x) = 0$$for any $x_0 \in M.$
\end{proposition}

\begin{proof}
In view of Corollary \ref{C:1.1}, we have $ess \, \inf_{x \in
B(x_0,r)} |\nabla f|(x) \le \frac{C_1D(M)}r\, ,$ for any $x_0 \in
M\, .$  Letting $r \to \infty$ gives the desired.
\end{proof}

\begin{example}$($\cite {HI}$)$ The expanding sphere
$\partial B_{r(t)}$ of radius $r(t)$
in $\mathbb{R}^n\, ,$ where
$$r(t) = e^{\frac t{n-1}}$$ satisfies the inverse mean curvature
flow $(\frac {\partial X_t}{\partial t} = \frac {\nu}{H})$, and
its level curve formulation satisfies  $\lim_{r \to \infty} ess \,
\inf_{x \in B(x_0,r)} |\nabla f|(x) = 0$ for any $x_0 \in \mathbb{R}^n.$
\end{example}

 \begin{proposition}\label{P:1.1}
Let $M$ be a complete noncompact manifold with the global doubling
property.
Then there does not exist a $W^{1,1}_{loc}$ weak solution of the
level set formulation of the inverse mean curvature flow \r{1.8} on $M $
such that  \r{1.11} holds for some $x \in M$.
\end{proposition}

\begin{proof} \quad
Suppose the contrary. Letting $r \to \infty$ in Corollary
\ref{C:1.1} would lead to a contradiction.\end{proof}



\begin{corollary}\label{C:1.2}
Let $\Omega \subset M\, ,$ and $f: \Omega \to
\mathbb{R}$ be a $W^{1,1}_{loc}$ weak solution of \r{1.8}. Then \r{1.10} and \r{1.12} hold.
\end{corollary}

\begin{proof} \quad  This is an immediate consequence of Theorem \ref{T:1.2} in which $A(x, f(x), \nabla
f(x)) = |\nabla f(x)|$, and $S_{\infty}(\Omega)  =
I_{\infty}(\Omega) \, $.\end{proof}


\begin{theorem}\label{T:3.1}
Let $\Omega$ be a domain in a Cartan-Hadamard manifold $M$ with
sectional curvature $Sec^M \le  - a^2\, ,$ where $a > 0\, .$ Then
there does not exist a $W^{1,1}_{loc}$ weak subsolution $f: \Omega
\to \mathbb{R}$  of \r{1.8}  with $w = ess \, \inf_{x\in \Omega}
|\nabla f(x)| > I_{\infty}(\Omega) > 0.$

\end{theorem}

\begin{proposition}\label{P:3.2}
Let $M$ be a complete noncompact manifold with $p$th-power volume
growth, $p \ge 0$. Let $f: M \to \mathbb{R}$ be a $W^{1,1}_{loc}$
weak subsolution of the level set formulation of the inverse mean
curvature flow \r{1.8} on $M$. Then
\begin{equation}\label{3.1} ess \, \inf_{x \in M} |\nabla f|(x) = 0\,.\end{equation}
\end{proposition}
In particular, if $M$ is a complete manifold of finite volume, or of nonnegative Ricci curvature, then \r{3.1} holds.

\begin{proof}
This follows at once from the last assertion of Proposition \ref{P:2.1}, and the Bishop-Gromov volume comparison theorem.
\end{proof}

\section{Liouville Theorem for
$p$-subharmonic functions of constant $p$-tension field}

\begin{theorem}\label{T:1.4}
Let $n \le p\, ,$ and $f\in W^{1,p}_{loc} (\mathbb{R}^n) \cap
C(\mathbb{R}^n)$ be bounded below. If $f$ is a weak $\emph
{subsolution}$ of the $p$-Laplace equation with constant $p$-tension
field, i.e., $div(|\nabla f|^{p-2}\nabla f) = c, \,$ and with
bounded weak derivative $($i.e., $|\nabla f| \le C_2$ for some
constant $C_2
> 0$$)$, then $f$ is constant.
\end{theorem}

\begin{proof}\quad
Let $\psi \ge 0$ be as in \cite {W1} Lemma 1 in which $M =
\mathbb{R}^n, t=r, s= \frac r2$. Choose $\psi$ to be a test
function  in the distribution sense of the constant $p$-tension
field equation. Then via the Cauchy-Schwarz inequality we have

\begin{align*}\begin{aligned}\notag & \int _{B(x_0,\frac r2)}
|c| \psi(x) dx  \le \int _{B(x_0,r)} |c| \psi(x) dx  \\&\hskip.3in = \bigg|
\int _{B(x_0,r)} |\nabla f|^{p-2}{\nabla f}\cdot \nabla \psi dx\bigg|
\le  C_2^{p-1} \int _{B(x_0,r)} |\nabla \psi| dx
.\end{aligned}\end{align*} Hence,
\[ |c| {\rm Vol}\left(B\left(x_0,\frac r2\right)\right) \le \frac{C_1 C_2^{p-1}}{r}{\rm Vol}
(B(x_0,r))\] implies that $c =0$ on letting $r \to \infty\, .$ Now
the assertion follows from Theorem 1.1 in \cite {BD}.\end{proof}

\begin{theorem}\label{T:1.5}
Let $n \le p\, .$ If $f\in W^{1,p} (\mathbb{R}^n) \cap
C(\mathbb{R}^n)$ is bounded below, and is a weak $\emph
{subsolution}$ of the $p$-Laplace equation with constant $p$-tension
field, then $f$ is constant.
\end{theorem}
\begin{proof}\quad Proceeding as in the
proof of Theorem \ref{T:1.4}, we have via the H\"older inequality

\begin{align*}\begin{aligned}\notag & \int _{B(x_0,\frac r2)}
|c| \psi(x) dx  \le \bigg| \int _{B(x_0,r)} |\nabla f|^{p-2}{\nabla
f}\cdot \nabla \psi dx\bigg|  \\&\hskip.3in \le  (\int _{\mathbb{R}^n}
|\nabla f|^p dx)^{\frac {p-1}{p}} \cdot (\int _{B(x_0,r)} |\nabla
\psi|^p dx)^{\frac 1p}.\end{aligned}\end{align*} Hence,
\[ |c| {\rm Vol}\left(B\left(x_0,\frac r2\right)\right) \le \frac{C_1 C_3}{r}{\rm Vol}
(B(x_0,r))\] where $C_3 \ge (\int _{\mathbb{R}^n} |\nabla f|^p
dx)^{\frac {p-1}{p}}$ is a constant independent of $r\, $(here we
use the assumption $f \in W^{1,p}(\mathbb{R}^n)$). Hence $c =0$ on
letting $r \to \infty\, ,$ and the assertion follows.\end{proof}

Theorems \ref{T:1.4} and \ref{T:1.5} also augment Theorem 1.1 in
\cite {BD}:\smallskip

\begin{theorem}\label{T:4.1} Let $n \le p\, .$ If $u\in
W^{1,p}_{loc} (\mathbb{R}^n) \cap C(\mathbb{R}^n)$ is a weak
supersolution of
$$\operatorname{div}(|\nabla f|^{p-2}\nabla f) = 0  \ \ in \ \ \mathbb{R}^n$$ and  is
bounded below then $u$ is constant.
\end{theorem}

As a further application of Theorem \ref{T:1.1} for $p$-subharmonic
functions of constant $p$-tension field for $p=1$, we have
\begin{proposition}\label{P:4.1} Let $M$ be a complete noncompact Riemannian manifold
with a doubling constant $D(M)$, and $f: M \to \mathbb{R}$ be a
$W^{1,1}_{loc}$ weak subsolution of the $1$-harmonic equation
\begin{equation}\label{4.1}  \operatorname{div} \bigg(\frac {\nabla g}
 { |\nabla g|}\bigg) = 0\, \ \ on \ \ M,\end{equation}
 with constant $1$-tension field $c\, ,$ i.e.,  \begin{equation}\label{4.2} \operatorname{div}
\bigg(\frac {\nabla f} { |\nabla f|}\bigg)= c\, \end{equation} in the
distribution sense. Then $f$ is a $1$-harmonic function.
\end{proposition}

\begin{proof} This follows at once from Theorem \ref{T:1.1} in
which $A(x, f, \nabla f) \equiv c$.
\end{proof}

This generalizes the case $M = \mathbb{R}^n$ of \cite {W1}.

\section{Generalized $1$-harmonic Functions on cocompact domains}


In \cite{W2}, we study $p$-harmonic geometry and related topics, and obtain, in particular:

\begin{theorem}\label{T:5.0} $([W2,\, \operatorname{Theorem}\, 9.6])$$\quad$  Let a complete manifold
$M$ have the global doubling property $\r{2.2}$, and $K$ be a
compact subset of $M$. If $g: M\backslash K \to \mathbb{R}$ is a
continuous function with $\inf _{x\in M\backslash K} |g(x)| > 0\,
,$  then there does not exist a $C^2$ solution $f$ of the equation
 $ {\rm div} \bigg(\frac {\nabla f} { |\nabla f|}\bigg)=
g(x)\, $ on $M\backslash K\, .$\end{theorem}
\smallskip

These ideas and techniques in studying nonhomogeneous $1$-harmonic functions in [W2, (9.1)] on the complement of a compact set on manifolds can provide a unified method in the following more general setting:

\begin{theorem}\label{T:5.1} Let $M$ and $K$ satisfy the assumptions in Theorem \ref{T:5.0}. Let $f\in W^{1,1}_{loc}(M)$ be a weak solution of
\begin{equation}\label{5.1} \operatorname{div} \bigg(\frac {\nabla f} { |\nabla f|}\bigg)= A(x,
f, \nabla f)\,\quad\quad \operatorname{on}\quad M \backslash
{K}\end{equation} where $A$ is a
continuous real-valued function on $M$, or $A$ is either nonnegative or nonpositive valued, with
\begin{equation}\label{5.2} w_{_{K^c}} := ess \, \inf_{x \in M \backslash K} |A(x,
f, \nabla f)|\end{equation} Then $w_{_{K^c}}=0$.\label{T:5.1}
\end{theorem}

\begin{lemma}\cite {W3}\label{L:5.1} Every complete manifold  $M$ with the global doubling
property \r{2.2} has infinite volume.
\end{lemma}

For completeness, we provide the following:

\begin{proof}[Proof of Lemma \ref{L:5.1}] Suppose on the
contrary  ${\rm Vol}(M) = V_0\, $ for some $0 < V_0 < \infty\, .$
Since $M$ is complete, $M = \cup_{i=0}^{\infty}B(x_0;2^i r)\, ,$
there would exist $n$ such that ${\rm Vol} (B(x_0;2^nr)) > \frac
{D(M)}{D(M)+1} V_0\, .$ Hence, ${\rm Vol} (M \backslash
B(x_0;2^nr)) < \frac{1}{D(M)+1} V_0\, .$ Choose $x_1 \in M$ such
that the distance between $x_0$ and $x_1$ is $3\cdot 2^n r\, ,$
Then $B(x_1; 2 \cdot 2^n r) \subset M \backslash B(x_0;2^nr)$ and
$ B(x_0;2^nr) \subset B(x_1; 4 \cdot 2^n r)\, .$ By the global
doubling property, this would lead the following contradiction:
$$\frac
{D(M)}{D(M)+1} V_0 < {\rm Vol} (B(x_1;4 \cdot 2^nr)) \le D(M) {\rm
Vol} (B(x_1;2 \cdot 2^nr)) < \frac{D(M)}{D(M)+1} V_0\, . $$
\end{proof}

As an immediate consequence of Lemma \ref{L:5.1}, one obtains the following result due to Calabi and Yau by different methods:
\begin{corollary}[\cite {Ca, Y}] Every complete manifold of nonnegative Ricci curvature has infinite volume.
\end{corollary}

\begin{proof}[Proof of Theorem \ref{T:5.1}] We consider two cases:

\noindent Case 1. $A(x,
f(x), \nabla f(x))$ is a continuous function on $M\backslash K\, :$ Then either it assumes both positive and negative values, in which case the assertion follows from the
intermediate value theorem, or we go to Case 2.

\noindent Case 2. $A(x,
f(x), \nabla f(x))$ is nonpositive or nonnegative valued: We proceed as in the proof of $[W2,\, \operatorname{Theorem}\, 9.6]\, .$   Since $K\subset M\, $ is compact,
choose a sufficiently large $r_0 < r$ such that $K\subset
B(x_0,r_0)\, .$
 Let $0 \le \psi \le 1$ be the cut-off function
as in \cite {W1} Lemma 1 in which $t=r, s= 2r$ (i.e. $\psi \equiv
1$ on the closure $\overline{B(x_0,r)}\, ,$ $\psi \equiv 0$ off
$B(x_0,2r)\, ,$ and $|\nabla \psi| \le \frac {C_1}{r}\,$ ). Multiplying both sides of \r{5.1} by $\psi\, ,$ integrating over
${B(x_0,2r)}\backslash {B(x_0,r)}\, ,$  and applying Stokes' theorem, we have
\begin{equation}\label{5.2}\begin{aligned} & w_{_{K^c}}({\rm Vol}(B(x_0,
r)) - {\rm Vol}(B(x_0, r_0))) \\ \le&  \bigg|\int _{B(x_0,r)\backslash B(x_0,
r_0)} A(x, f(x), \nabla f(x))\psi(x) dx\bigg| \\\le&  \bigg| \int
_{B(x_0,2r)\backslash B(x_0, r_0)} A(x, f(x), \nabla f(x))\psi(x)
dx\bigg|
\\ =& \bigg| \int _{\partial B(x_0,r_0)}  \frac {\nabla f} { |\nabla
f|} \cdot \nu dS +  \int _{B(x_0,2r)\backslash B(x_0, r_0)}
 \ \ \frac {\nabla f} { |\nabla f|}\cdot \nabla \psi
dx \bigg|
\\\le&    {\rm Vol}(\partial B(x_0, r_0)) +  \frac {C_1 {\rm Vol}(B(x_0,
2r))}r\, ,
\end{aligned}\end{equation} where $\nu$  is  the  unit  normal  to  $\partial B(x_0,r)\, ,$ $dS$ is the area element of $\partial B(x_0,r)\, ,$ and $C_1 >0$ is the constant as above.
Hence, dividing \r{5.2} by ${\rm Vol}(B(x_0, r))\, ,$ and using
\r{5.2} and \r{2.2}, one has
\[ w_{_{K^c}}  (1-\frac {{\rm Vol}(B(x_0,
r_0)}{{\rm Vol}(B(x_0, r))})  \le \frac {{\rm Vol}(\partial B(x_0,
r_0))}{{\rm Vol}(B(x_0, r))} + \frac { C_1 D(M)}{r}\to 0\]
as $r \to \infty\, ,$
since by Lemma \ref{L:5.1}, $M$ has infinite volume.
\end{proof}

\begin{remark}\label{R:5.1}
From the above proof, it is clear that if we weaken the assumption that $K$ is compact in Theorem \ref{T:5.1} to \emph {$K$ being precompact or bounded}, the conclusion in Theorem \ref{T:5.1} remains true. Similarly, the following Proposition \ref{P:5.1} and Corollary \ref{C:1.3} also hold on the complement of any precompact or bounded set $K\, .$
\end{remark}
As an immediate consequence of Theorem \ref{T:5.1},
we have:

\begin{proposition} \label{P:5.1}
Let $K\, ,$ $M\, ,$ and $w\, $ be as in Theorem \ref{T:5.1}. Then there does not
exist a generalized $1$-harmonic function on the complement of a compact set $K$ with $w_{_{K^c}} > 0$.
\end{proposition}

\begin{proof}
This follows immediately from Theorem \ref{T:5.1}.
\end{proof}

\begin{corollary}\label{C:1.3} On the complement of any compact subset $K$ in a complete manifold $M$ of nonnegative Ricci curvature, there does not exist a
weak solution of the level set formulation of the inverse mean
curvature flow \r{1.8}, with constant $1$-tension field \r{4.2}.
\end{corollary}

\begin{proof}
Suppose, on the contrary, that there were such a weak solution. Then by Theorem \ref{T:5.1}, \r{1.8} and \r{4.2}, we would have $w_{_{K^c}}=0=c=|\nabla f|\, ,$ i.e., $f= constant\, .$ This contradicts
our remark in the introduction that a constant is not a weak solution of \r{1.1}.
\end{proof}

\section{Nonexistence of solutions of the Initial value problem for Inverse Mean Curvature Flow}
In the section, we begin by combining the definition of a weak solution $\big ($in the distribution sense; cf. \r{1.2}$\big )$
with an initial condition consisting of a bounded open set $E_0$ with a boundary at least $C^1\, .$
We say that a function $f$ is  {\it a weak solution
of the level set formulation of
the inverse mean curvature flow
\begin{equation}\label{6.1} \operatorname{div} \bigg(\frac {\nabla f}
{ |\nabla f|}\bigg)= | \nabla f |\, \end{equation}
with initial condition $E_0\, $ if $$f \in W^{1,1}_{loc}(M)\, , E_0 =\{x : f(x)<0\}\, , \operatorname{and}\, f\,  \operatorname{is}\,  \operatorname{a}\,  \operatorname{weak}\,  \operatorname{solution}\,  \operatorname{of}\,  (6.1)\,  \operatorname{in} M\backslash E_0\, .$$

All the local estimates \r{1.4}, \r{1.6}, and \r{1.9} (cf. Theorems \ref{T:1.1} and  \ref{T:1.2}, and Corollay \ref{C:1.1}) hold for $\Omega = M\backslash E_0\, .$  On the other hand, we have the following global result for the initial value problem for inverse mean curvature flow, which is in contrast to Proposition \ref{P:3.1} in which $\Omega = M$.

\begin{theorem}\label{T:6.1} Let $M$ be a complete manifold with the global doubling
property. Let $f$ be a weak solution
of the level set formulation of
the inverse mean curvature flow
 \r{6.1} with initial condition $E_0\, .$  Then $$ ess \, \inf_{x
\in M\backslash E_0} |\nabla f|(x) = 0\, .$$
\end{theorem}

or equivalently,

\begin{theorem}\label{T:6.2}$[$Nonexistence of Weak Solutions$]$ Let $M$ be a complete manifold with the global doubling
property. If $$ ess \, \inf_{x
\in M\backslash E_0} |\nabla f|(x) >0\, ,$$ then there does not exist a weak solution of
the level set formulation of
the inverse mean curvature flow \r{6.1} with initial condition $E_0\, .$
\end{theorem}

\begin{proof}[Proof of Theorems \ref{T:6.1} and \ref{T:6.2}] This follows immediately from
Theorem \ref{T:5.1} and Remark \ref{R:5.1}. \end{proof}

\begin{corollary}\label{C:6.1} There does not exist a
weak solution of \r{6.1} with initial condition $E_0\, ,$ and with constant $1$-tension field on a
complete manifold $M$ of nonnegative Ricci curvature.
\end{corollary}

\begin{proof}[Proof of Corollary \ref{C:6.1}]
This follows at once from Proposition \ref{P:5.1}, Remark \ref{R:5.1}, and the fact
that a manifold of nonnegative Ricci curvature has the global
doubling property.
\end{proof}


\begin{thebibliography}{99}

\bibitem[BD]{BD}  I. Birindelli and F. Demengel, {\it Some Liouville
theorems for the $p$-Laplacian}, Proceedings of the
  2001 Luminy Conference on Quasilinear Elliptic and Parabolic Equations and System, 35--46 (electronic), Electron. J. Differ. Equ. Conf., 8, Southwest Texas State Univ., San Marcos, TX, 2002

\bibitem[BDG]{BDG} E. Bombieri, E. de Giorgi, E. Giusti, Minimal cones and the
Bernstein problem, Invent. Math. 7 (1969), 243-268

\bibitem[BN]{BN} H.  Bray and A. Neves, {\it Classification of prime 3-manifolds with Yamabe invariant greater than $RP^3$}, Annals of Mathematics,159, (2004), 407-424

\bibitem[Ca]{Ca} E. Calabi, {\it On manifolds with nonnegative Ricci curvature II}, Notices Amer. Math. Soc. (1975)A205.

\bibitem[C]{C} J. Cheeger, {\it A lower bound for the smallest eigenvalue of
the Laplacian}, Problems in analysis (Papers dedicated to Salomon
Bochner, 1969), pp. 195--199. Princeton Univ. Press, Princeton, N.
J., 1970

\bibitem[CGG]{CGG} Y. G. Chen, Y. Giga, and S. Goto, {\it Uniqueness and existence of viscosity solutions of generalized mean curvature flow equations}, J. Differential Geom. 33 (1991), no. 3, 749--786

\bibitem[D]{D} F. Demengel, {\it Functions locally almost $1$-harmonic}, Applicable Analysis, 83, (2004), no. 2, 865--896

\bibitem[DW]{DW} Y. X. Dong and S. W. Wei, {\it On vanishing theorems for vector bundle valued $p$-forms and their applications}, Comm. Math. Phy. (in press). arXive: 1003.3777v2

\bibitem[ES]{ES} L. S. Evans  and J. Spruck, {\it Motion of level sets by mean curvature I}, J. Differential Geom. 33 (1991), no. 3, 635--681.

\bibitem[FF]{FF} H. Federer and W. H. Fleming, {\it Normal and integral
currents}, Ann. of Math. (2) 72 1960 458--520

\bibitem[HI]{HI} G. Huisken and T. Ilmanen,  {\it The inverse mean
curvature flow and the Riemannian Penrose inequality}, J.
Differential Geom. 59 (2001), no. 3, 353--437

\bibitem[KS]{KS} B. Kawohl and F. Schuricht, {\it Dirichlet problems for the 1-Laplace operator, including the eigenvalue problem}, Commun. Contemp. Math. 9 (2007), no. 4, 515--543

\bibitem[KN]{KN} B. Kotschwar and L. Ni,
{\it Local gradient estimates of p-harmonic functions, 1/H-flow,
and an entropy formula},  Ann. Sci. $\acute{E}$c. Norm. Sup$\acute{e}$r. (4) 42 (2009), no. 1, 1--36

\bibitem[L]{L} E. B. Lin, {\it On nonhomogeneous $A$-harmonic equations and 1-harmonic equations},
J. Inequal. Appl. 2010, Art. ID 346308, 4 pp.

\bibitem[M]{M} R. Moser, {\it The inverse mean curvature flow and $p$-harmonic functions}, J. Eur. Math. Soc. (JEMS) 9 (2007), no. 1, 77--83.

\bibitem[WW]{WW} S. P. Wang and S. W. Wei, {\it Bernstein conjecture in hyperbolic geometry
in Seminar on Minimal Submanifolds}, Editor E. Bombieri, Ann. of
Math. Stud., no. 103, 1983, 339--358

\bibitem[W1]{W1} S. W. Wei, {\it  On 1-harmonic functions}, SIGMA Symmetry Integrability Geom. Methods Appl. 3 (2007), Paper 127, 10 pp. arXiv:0712.4282

\bibitem[W2]{W2}  S.W. Wei, {\it $p$-harmonic geometry and related topics}. Bull. Transilv. Univ. Brasov Ser. III 1(50) (2008), 415-453

\bibitem[W3]{W3} S. W. Wei, {\it  The unity of $p$-harmonic geometry}, preprint

\bibitem[WLW]{WLW} S. W. Wei, J. F. Li and L. N. Wu, {\it Generalizations of the uniformization
theorem and Bochner's method in $p$-harmonic geometry}, Proceedings of the 2006 Midwest Geometry Conference, Commun.
Math. Anal. 2008, Conference 1, 46-68.

\bibitem[Y]{Y} S. T. Yau, {\it Some function-theoretic properties of complete Riemannian manifold and their applications to geometry}, Indiana Univ. Math. J. 25 (1976), no. 7, 659--670;  Erratum: Indiana Univ. Math. J. 31 (1982), no. 4, 607
\end{thebibliography}
\end{document}